\documentclass{amsart}
\usepackage{graphicx}
\renewcommand{\Re}{\mathrm{Re}}
\newcommand{\Cl}{\mathrm{Cl}}
\newcommand{\Z}{\mathbb Z}

\newtheorem{theorem}{Theorem}[section]
\newtheorem{conjecture}[theorem]{Conjecture}
\newtheorem{definition}[theorem]{Definition}
\newtheorem{lemma}[theorem]{Lemma}
\newtheorem{corollary}[theorem]{Corollary}
\newtheorem{proposition}[theorem]{Proposition}
\newtheorem{fact}[theorem]{Fact}
\newtheorem{counter}[theorem]{Counterexample}
\newtheorem{question}[theorem]{Question}
\begin{document}

\markboth{V. A. Manathunga, J. Conant}
{The Conway Polynomial and Amphicheiral Knots}

%%%%%%%%%%%%%%%%%%%%% Publisher's Area please ignore %%%%%%%%%%%%%%
%\catchline{}{}{}{}{}
%%%%%%%%%%%%%%%%%%%%%%%%%%%%%%%%%%%%%%%%%%%%%%%%%%%%%%%%%%%%%%%%%%%

\title{THE CONWAY POLYNOMIAL AND AMPHICHEIRAL KNOTS}

\author{J. Conant}

\address{Department of Mathematics \\
University of Tennessee\\
Knoxville, TN 37996\\
jconant@math.utk.edu}

\author{V. A. Manathunga}

\address{Department of Mathematics and Statistics\\
Austin Peay State University,\\
 Clarksville TN 37040 \\
manathungav@apsu.edu}

\begin{abstract}
According to work of Hartley and Kawauchi in 1979 and 1980, the Conway Polynomial of all negative amphicheiral knots and strongly positive amphicheiral knots factors as $\phi(z)\phi(-z)$ for some $\phi(z)\in\Z[z]$. Moreover, a 2012 example due to Ermotti, Hongler and Weber shows that this is not true for general amphicheiral knots. On the other hand, in 2006 the first author made a conjecture equivalent to saying that the Conway polynomial of all amphicheiral knots splits as $\phi(z)\phi(-z)$ in the ring $\mathbb Z_4[z]$. In this paper, we establish this conjecture for all periodically amphicheiral knots built from braids, where the period preserves the braid structure. We also 
give counterexamples to conjectures on the leading coefficient of the Conway polynomial of an amphicheiral knot due to Stoimenow.
\end{abstract}
\maketitle
%\keywords{Conway polynomial, Alexander polynomial, amphicheirality}

%\ccode{Mathematics Subject Classification 2000: 57M25, 57M27}

\section{Introduction}	%) A SECTION HEADING
 A knot is called amphicheiral if it coincides with its mirror image. More formally, we say there is an orientation reversing map from $S^3 \to S^3$ which fixes the knot setwise. In the category of oriented knots, we can distinguish this amphicheirality as positive or negative depending on whether the map preserves or reverses the orientation of the knot. If the map used in this process is an involution, we call the resulting amphicheirality ``strong''. It is worthwhile to note that a knot can be positive and negative amphicheiral at the same time \cite{bur1}. The figure eight is an example of such a knot.
%  For almost 100 years, knot theorists believed that the minimal crossing number of an amphicheiral knot is even \cite{St2}. However, with availability of computer technology, M. Thistlethwaite in 1990 shown that this is false. Now it has been proven that for each odd natural $n$ number greater than $14$ there exists an amphicheiral knots of crossing number $n$ \cite{St2}.  

The Jones Polynomial is a powerful tool for detecting chirality, since the Jones Polynomial of a mirror image is often different from that of the original knot. The Conway Polynomial is, on the other hand, invariant under the process of taking a mirror image, so it not obvious that it could be used for detecting chirality.
However, one can indeed use the Conway (or Alexander) polynomial to detect chirality among knots. In this article, we discuss two conjectures, one proposed by the first author and one by A. Stoimenow. 

The first conjecture (Conjecture \;\ref{conj11p}) has its foundation in the theory of Vassiliev invariants. However, after reformulation of the conjecture (Conjecture \;\ref{conj11}) it can be seen that it closely resembles a theorem of A. Kawauchi and R. Hartley \cite{hartley2}. Their theorem states that for negative and strong positive amphicheiral knots, the Conway polynomial splits. That is, the Conway polynomial of these two types of  amphicheiral knots can be written as $f(z)f(-z)$ for some integer polynomial $f$. However this is not true for the remaining types of amphicheiral knots, positive amphicheiral knots which are not strong, and counterexamples have been found \cite{ermotti1}. The first author's conjecture claims that instead of considering splittings over $\Z$, if we consider splittings over $\Z_4$ instead, then in fact we can write the Conway polynomial of any amphicheiral knot as $f(z)f(-z)$ for some $f\in \Z_4[z]$, including the nonstrong positive amphicheiral case \cite{jconant1}. In general, to prove the conjecture for all amphicheiral knots, one needs only to prove it for hyperbolic knots (Theorem~\ref{theorem13}).

Now, the symmetry groups for hyperbolic knots are classified: they are either dihedral or cyclic\cite{host,kod,ril}. If the knot is amphicheiral, they have to contain an orientation-reversing element. If there is an orientation-reversing element of order 2, that means the map is involution, so we know the knot is  strongly amphicheiral and the Conway polynomial of the knot splits as  $f(z)f(-z)$ by Hartley-Kawauchi. Dihedral groups are generated by involutions, so any amphicheiral knot with dihedral symmetry group must have an orientation-reversing involution. So we are left with cyclic groups $\Z_{2n}$ with orientation reversing generator $h$. If $n$ is odd, then $h^n$ is an orientation-reversing involution, so we are done in that case. Hence we may assume $n$ is even. All hyperbolic knots with these cyclic symmetry groups can be constructed by taking a tangle T, considering the concatenation $(T\cdot T^*))^n$ and closing it up via the standard closure. Here $T^*$ is the mirror image tangle with all crossings reversed. In sections 2 and 3, we consider the case where the tangle is actually a braid, proving Conjecture \ref{conj11} is true for all amphicheiral knots coming from braids of the form $(ww^*)^{n}$. In this case, we are able to use the Burau representation and its connection to the knot's Alexander polynomial. In the case of a general tangle $T$, this particular method is unavailable, so a new technique is needed. It is worthwhile to note that there are amphicheiral knots which are not periodic \cite{hartley3}.

%For a general tangle $T$, the problem becomes much harder due to lack of methods for calculating the Conway polynomial of a tangle $(T.m(T))^n$. 

The second conjecture (Conjecture~\ref{conj13}) is based on the leading coefficient of the Conway polynomial. The leading coefficient of the Conway polynomial can sometimes be used to detect chirality among knots. In 1997, K. Murasugi and J. Prztycki proved \cite{Mu} that if the leading coefficient of the Conway polynomial of an  alternating knot is prime, then it is not amphicheiral. Another implication of A. Kawauchi and R. Hartley's theorem discussed above is that except for the non strong positive amphicheiral case, the leading coefficient of the Conway polynomial of an amphicheiral knot is square. These results lead to the conjecture that the leading coefficient of amphicheiral knots is square or at least not prime. A. Stoimenow further strengthens this conjecture to hypothesize that the leading coefficient of the Conway polynomial is square for certain other amphicheiral knot classes such as the classes of alternating amphicheiral knots. In section 4, we give several counterexamples to this conjecture by constructing nonstrong positive amphicheiral knots as braid closures. 

\section{Amphicheiral knots and the Conway polynomial}

Considerations in Vassiliev theory led the first author  \cite{jconant1} to conjecture

\begin{conjecture}\label{conj11p}
If $K$ is an amphicheiral knot, then its Conway polynomial $C(z)$ satisfies the property that $C(z)C(iz)C(z^2)$ is a perfect square in $\Z_4[z^2]$
\end{conjecture}

Our first task is to give a simpler reformulation of this conjecture. 
\begin{theorem}\label{theorem11}
Let $C(z)$ denote the Conway polynomial of a knot $K$. Then the following are equivalent.
\begin{enumerate}
%\item $pc_{4i}\equiv 0\mod 2$
\item $C(z)C(iz)C(z^2)$ is a perfect square in $\Z_4[z^2]$
\item $C(z)C(iz)\equiv C(z^2) \mod 4$
\item $C(z)$ splits in $\mathbb{Z}_4[z]$. I.e. $C(z)\equiv f(z)f(-z) \mod 4$ for some $f \in \Z[z]$
\end{enumerate}
\end{theorem}
\begin{proof}
%In \cite{jconant1}, the equivalence of (1) and (2) is proven. We can easily see that (3) implies (2).

To prove (1) implies (2), suppose $C(z)C(iz)C(z^2)=f(z)^2\in \Z_4[z^2]$. 
%The Conway polynomial can be written as, $C(z)=1+\Sigma a_{2i+1}z^{2(2i+1)}+\Sigma a_{2i}z^{4i}$. 
First we consider what happens modulo 2. Note that $g(z)^2=g(z^2)\in \mathbb Z_2[z]$, and that $C(iz)=C(z)\in   \mathbb Z_2[z]$.
So 
$$
C(z)C(iz)C(z^2)=C(z)C(z)C(z^2)=C(z^2)C(z^2)
$$
But square roots are unique in $\mathbb Z_2[z]$, so $f(z)\equiv C(z^2)\mod 2$, and we have that $f(z)=C(z^2)+2g(z)$ for some $g(z)\in\mathbb Z_4[z]$. Squaring both sides, we see that
$f(z)^2=C(z^2)^2\in\mathbb Z_4[z]$. Multiplication by $C(z^2)$ is an injection on $\mathbb Z_4[z]$, so the equation 
$$C(z)C(iz)C(z^2)=f(z)^2=C(z^2)^2$$
implies that $C(z)C(iz)=C(z^2)\in\mathbb Z_4[z]$, as desired.

We can easily see that (2) implies (1).

%Let $y=z^2$. Then, $C(y)=1+f_o(y)+f_e(y)$ where $f_o$ denotes odd power terms and $f_e$ denotes even power terms. Similarly we denote $C(iz)$ using the notation $C(y)^*$. Now, $C(iz)=C(y)^*=1-f_o(y)+f_e(y)$. 
%By assumption, $C(z)C(iz)C(z^2)=C(y)C(y)^*C(y^2)=F^2$ for some $F\in \Z_4[y]$. Projecting to $\Z_2$, we have $\bar{C}(y)\bar{C}(y)^*\bar{C}(y^2)=\bar{F^2}\in \Z_2[y]$. Note that in $\Z_2$, $\bar{C}(y)=\bar{C}(y)^*$. So, $\bar{C}^2(y)\bar{C}(y^2)=\bar{F}^2\in\Z_2[y]$. This implies, $\bar{C}^4(y)=\bar{F}^2(y)$ and so $\bar{C}^2(y)=\bar{F}(y)$. So we can see, $F=C^2(y)+2g(y)\in \Z_4(y)$, for some $g\in \Z_4[y]$. Thus $F^2=C^4(y)\in \Z_4[y]$. Since $F^2=C^4(y)=C(y)C(y)^*C(y^2)$, we conclude that (??proof left cancellation??), $C(y)^*C(y^2)=C^3(y)\in \Z_4[y]$. Now observe that, $(C(y)^*)^2=C^2(y)$. Therefore, we can conclude that $C(y)^*C(y^2)=(C(y)^*)^2C(y)$ which implies $C(y^2)=C(y)C(y)^*\in \Z_4[y]$. This implies that $C(z)C(iz)\equiv C(z^2)\mod 4$. 

To prove (2) implies (3), substitute $z$ for $z^2$ in (3), then  $C(z)\equiv f(z)f(-z)\mod 4$ for some integer polynomial $f$. 

To prove (3) implies (2) we use following argument. Suppose $C(z)= f(z)f(-z)$ for some $f\in \Z[z]$. Let $\alpha(z),\beta(z),\gamma(z),\delta(z)\in \Z[z^4]$. Then a general integer polynomial $f(z)$ can be written $f(z)=\alpha(z)+\beta(z)z+\gamma(z)z^2+\delta(z)z^3$. 
\begin{align*}
&f(z)f(-z)f(iz)f(-iz)=\\
&=(\alpha+\beta z+\gamma z^2+\delta z^3)(\alpha-\beta z+\gamma z^2-\delta z^3)(\alpha +i\beta z-\gamma z^2-i\delta z^3)(\alpha -i\beta z-\\ &\gamma z^2+\delta z^3)\\
&=[(\alpha+\gamma z^2)^2-(\beta z+\delta z^3)^2][(\alpha-\gamma z^2)^2-(i\beta z-i\delta z^3)^2]\\
&=[(\alpha+\gamma z^2)^2-(\beta z+\delta z^3)^2][(\alpha-\gamma z^2)^2+(\beta z-\delta z^3)^2]\\
&\equiv [(\alpha+\gamma z^2)^2-(\beta z+\delta z^3)^2][(\alpha+\gamma z^2)^2+(\beta z+\delta z^3)^2]\\
&=(\alpha+\gamma z^2)^4-(\beta z+\delta z^3)^4\\
&=F(z)^4-G(z)^4
\end{align*}
Here we have used the fact $(A+B)^2\equiv (A-B)^2\mod 4$. On the other hand
\begin{align*}
f(z^2)f(-z^2)&=(\alpha(z^2)+\beta(z^2)z^2+\gamma(z^2)z^4+\delta(z^2)z^6)(\alpha(z^2)-\beta(z^2)z^2+\gamma(z^2)z^4- \\& \delta(z^2)z^6)\\
&=(\alpha(z^2)+\gamma(z^2)z^4)^2-(\beta(z^2)z^2+\delta(z^2)z^6)^2\\
&=F(z^2)^2-G(z^2)^2
\end{align*}
So, to show that $f(z)f(-z)f(iz)f(-iz)\equiv f(z^2)f(-z^2)\mod 4$. It suffices to show that $F(z^2)^2\equiv F(z)^4 \mod 4$ for any integer polynomial $F(z)$. This can easily be proven by induction. The base case of a monomial is trivial. Otherwise, write $F(z)=A(z)+B(z)$ for two polynomials of shorter length. Then 
\begin{align*}
F(z^2)^2&=A(z^2)^2+2A(z^2)B(z^2)+B(z^2)^2\\
&=A(z)^4+2A(z)^2B(z)^2+B(z)^4\\
&=(A(z)+B(z))^4
\end{align*}
 using the fact that $(A+B)^4\equiv A^4+2A^2B^2+B^4 \mod 4$.
  Thus we proved that if $C(z)=f(z)f(-z)\mod 4$ then, 
\begin{align*}
C(z)C(iz)&\equiv f(z)f(-z)f(iz)f(-iz) \mod 4\\
		 &\equiv f(z^2)f(-z^2)\mod 4\\
		 &\equiv C(z^2) \mod 4
\end{align*}

\end{proof}

Thus Conjecture~\ref{conj11p} can be reformulated as
\begin{conjecture}\label{conj11}
If $K$ is amphicheiral, then $C(z)\equiv f(z)f(-z) \mod 4$ for some $f \in \Z[z]$
\end{conjecture}

Surprisingly, Conjecture \ref{conj11} is very close to a theorem proved by A. Kawauchi and R.I. Hartley :
\begin{theorem}\label{theorem12}
\cite{hartley1, hartley2, kawauchi1}Let $C(z)$ be the Conway polynomial of a knot $K$. If $K$ is a negative amphicheiral knot then $C(z)=f(z)f(-z)$ for some $f\in \Z[z]$. If $K$ is a strongly positive  amphicheiral knot then $C(z)=f(z^2)^2$ for some $f\in \Z[z]$.  
\end{theorem}
This proves that Conjecture \ref{conj11} is true for all $(-)$ and strongly $(+)$ amphicheiral knots. From Hartley and Kawauchi's theorem, it is natural to formulate the following strengthening of Conjecture~\ref{conj11}.
\begin{conjecture}\label{conj12}
For any amphicheiral knot $K$, $C_K(z)=\phi(z)\phi(-z)$ for some $\phi \in \mathbb{Z}[z]$.
\end{conjecture}
However, this conjecture is {\bf false.} A counterexample was found in 2012 by N. Ermotti, C. V. Q. Hongler, C. Weber \cite{ermotti1}.  Figure~\ref{fig11} reproduced from \cite{ermotti1} shows this counterexample knot.
\begin{figure}[h!]
	\centering
	\includegraphics[width=6cm]{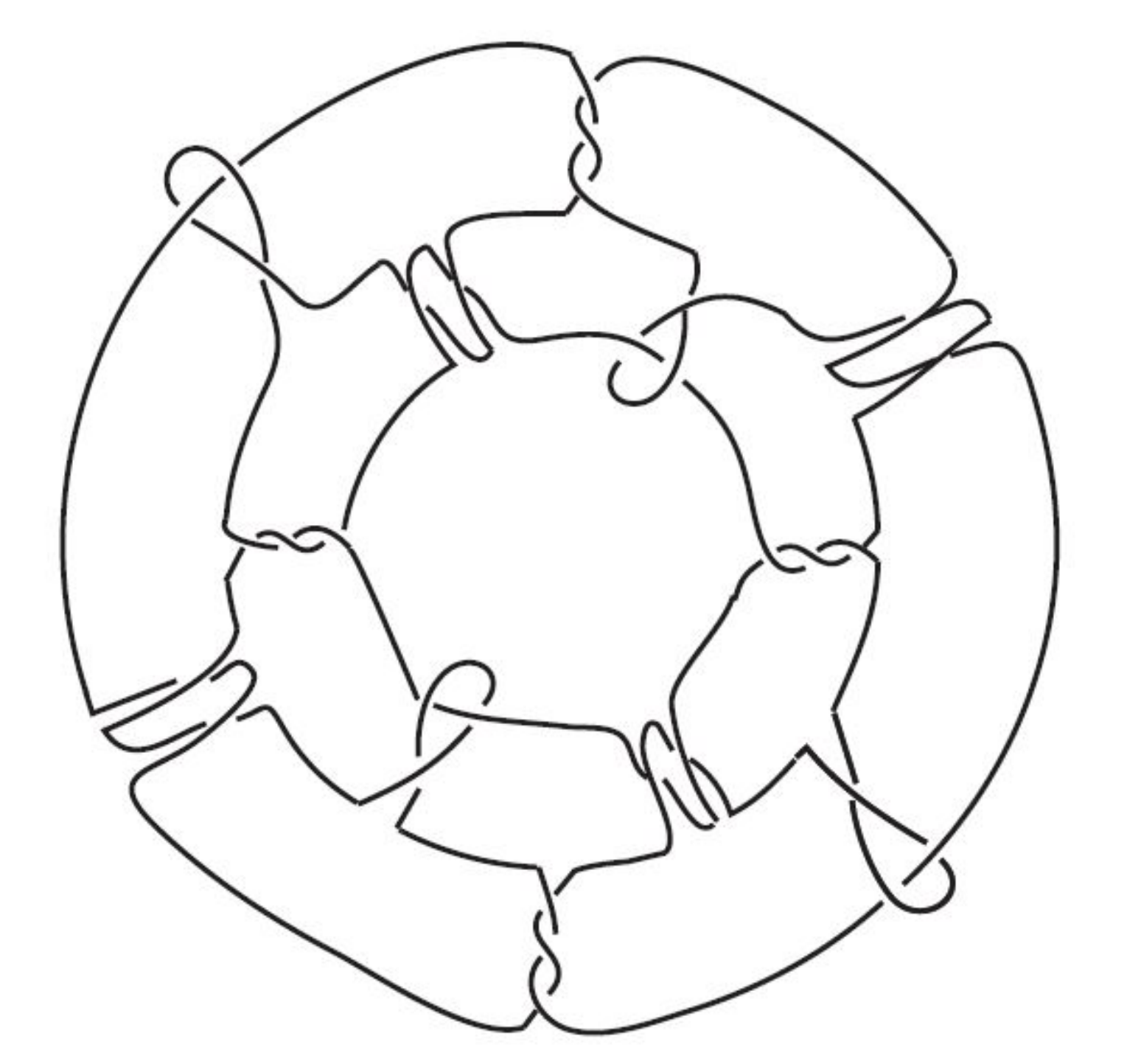}
	  \caption{The Ermotti-Hongler-Weber knot}\label{fig11}
\end{figure}
This knot is an alternating positive amphicheiral knot with Conway polynomial $C(z)=(4z^8+16z^6+12z^4-16z^2+1)(1 + z)(1-z)(2z^{4}-1)^2$ which does not split as $f(z)f(-z)$ for any $f(z)\in \mathbb{Z}[z]$. Thus it is indeed a counterexample to Conjecture \ref{conj12}. However note that $C(z)\equiv(1+z)(1-z)(2z^4-1)^2\equiv \phi(z)\phi(-z)\mod 4$. 

We have now produced many counterexamples to Conjecture \ref{conj12}. A nonalternating counterexample is given in Figure \ref{fig12}.
\begin{figure}[h!]
	\centering
	\begin{minipage}{5cm}\includegraphics[width=5cm]{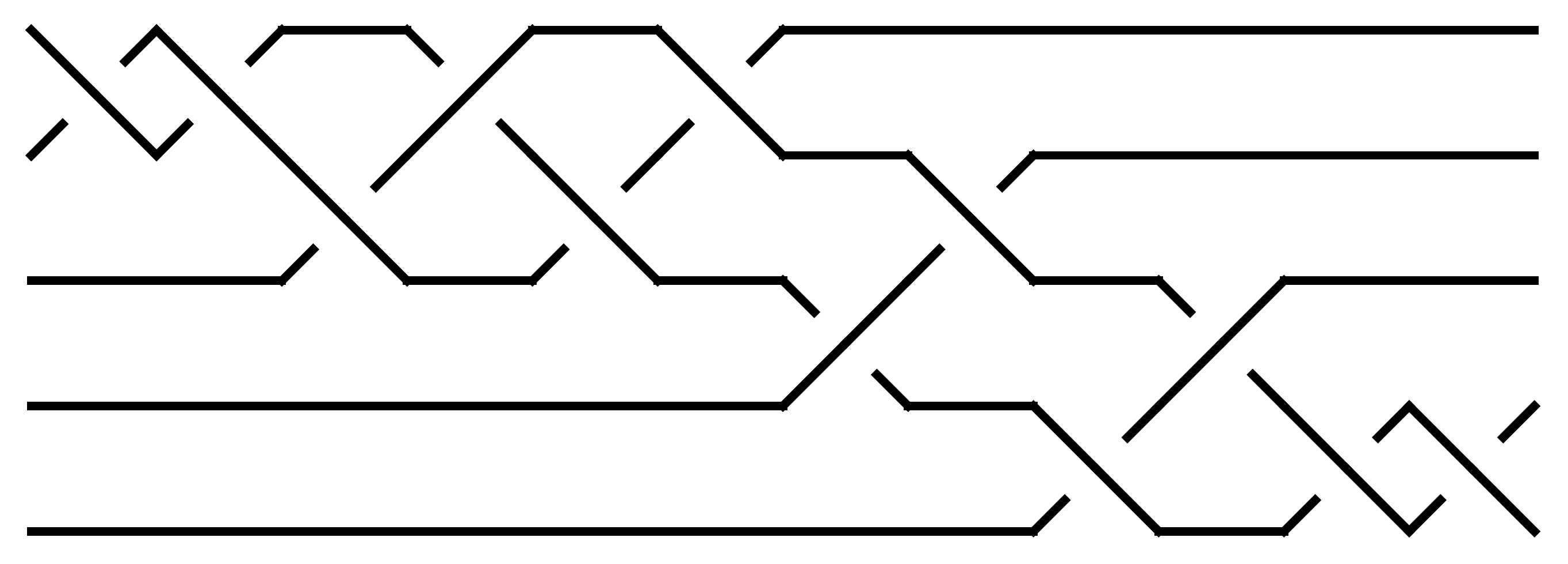}
	\end{minipage} \,\,\,\,\, \begin{minipage}{5cm}\includegraphics[width=5cm]{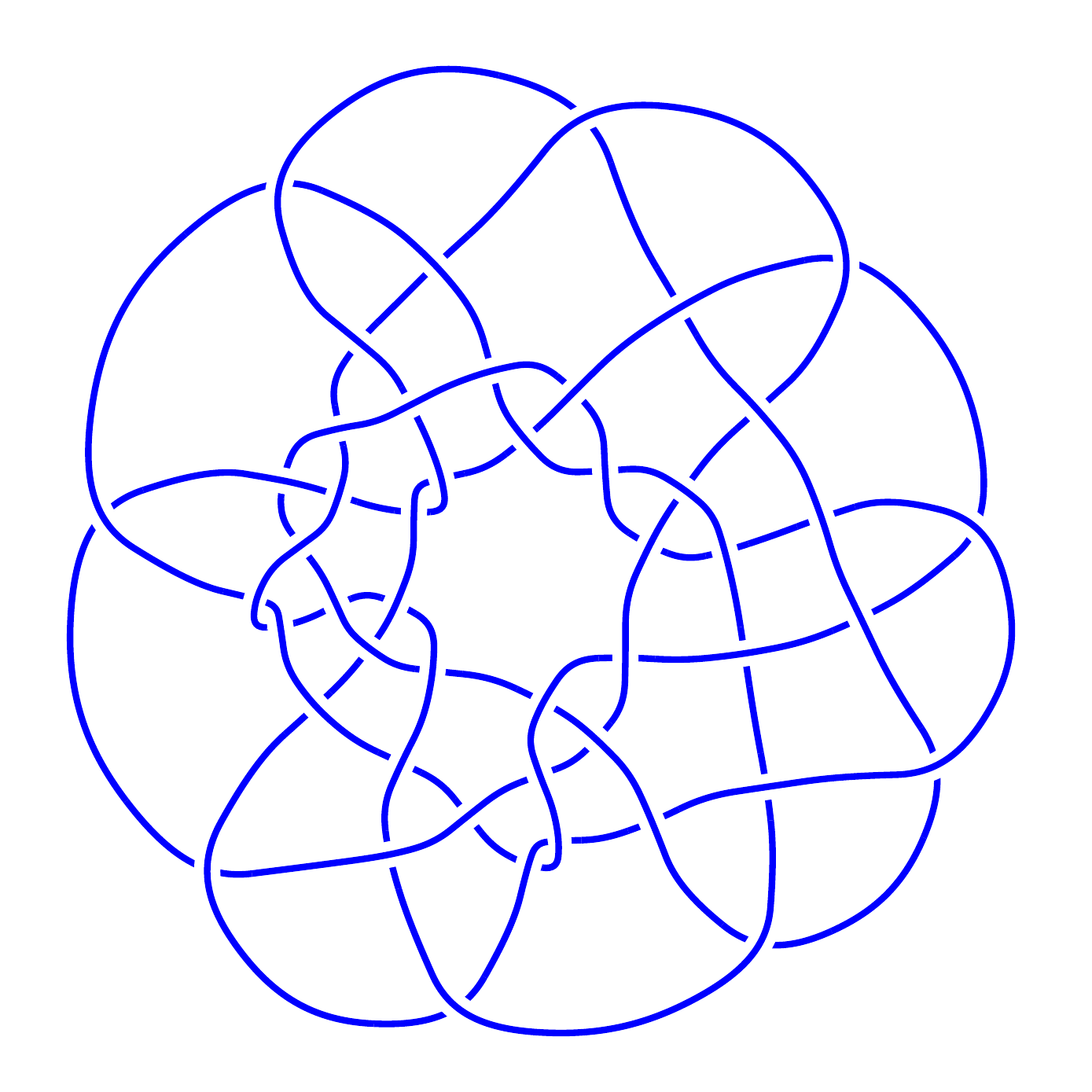}\end{minipage}
	\caption{The braid $w$ and its closure}\label{fig12}
\end{figure}
Let $w=\sigma_1^2\sigma_2\sigma_1^{-1}\sigma_2\sigma_1\sigma_3^{-1}\sigma_2\sigma_4\sigma_3^{-1}\sigma_4^2$ be the 5-strand braid in Figure \ref{fig12}, then the Conway polynomial of the closure of $ww^*ww^*$ is: $C_K(z)=(1+z^2)^2(1+3z^2)^2(1-11z^2+33z^4+8z^6)$ where $w^*$ is the braid $w$ with all crossings reversed. Thus it is a counterexample to the Conjecture \ref{conj12} because braids of this form are amphicheiral (see section \ref{sect1}). The knot is nonalternating due to the fact that the absolute value of the leading coefficient of the Conway polynomial of an alternating amphicheiral knot is square \cite{St1}. Still, this is not a counterexample to Conjecture \ref{conj11} because, $C_K(z)\equiv(1-z+z^2)(1+z+z^2)(1+z^4)^2 \mod 4$. 

%Now observe that $pc_{4i}$ invariants are primitive Vassiliev invariants. 

%We note that to prove Conjecture \ref{conj11}  one can focus only on prime amphicheiral knots. Due to theorem \ref{theorem12}, one needs to consider only prime, positive (non-strong) amphicheiral knots.  Also due to the following theorem  one can focus only on prime periodically positive amphicheiral knots.

Hartley \cite{hartley1} shows that $C_K(z)=\phi(z)\phi(-z)$ for all negative amphicheiral knots. His proof takes as input his joint result with Kawauchi \cite{hartley2}, that strongly negative amphicheiral knots have this property. %Any hyperbolic negative amphicheiral knot is actually strongly negative amphicheiral. 
Hartley then proves the general case by breaking up the knot complement into pieces which are hyperbolic or Seifert fibred via the JSJ decomposition. 
%Seifert fibred pieces don't contribute because the only Seifert fibred knot complements come from torus knots, which are never amphicheiral. %The splicing formula for Alexander polynomial is essential to show that the JSJ decomposition will play well with the splitting of $C_K(z)$.
One can mimic this argument to show that Conjecture ~\ref{conj11} follows if it can be proven for hyperbolic positive amphicheiral knots. 

\begin{theorem}\label{theorem13}
If $C_K(z)=\phi(z)\phi(-z)\mod 4$ for all hyperbolic positive amphicheiral knots then $C_K(z)=\phi(z)\phi(-z)\mod 4$ for all amphicheiral knots. 
\end{theorem} 
\begin{proof}
First we note that the condition that $C_K(z)=\phi(z)\phi(-z)\mod 4$ is equivalent to the condition that $\Delta_{K}(t^2)=s(t)s(t^{-1})\in\mathbb Z_4[t^{\pm 1}]$ where $s(t)=s(-t^{-1})$. So we may as well work with Alexander polynomials. The proof of Theorem 3.1 of \cite{hartley1} shows that the Alexander polynomial of an amphicheiral knot $K$ can be written as a finite product
$$\Delta_K(t)=\prod_{i} (g_i(t^{\alpha_i}))^{\gamma_i}$$
where each $g_i(t)$ is the Alexander polynomial of a knot $K_i$. If $\gamma_i$ is odd, then the knot $K_i$ is hyperbolic and amphicheiral and $\alpha_i$ is either $0$ or odd. By hypothesis, the Alexander polynomials in this odd case split over $\mathbb Z_4$, and the proof in \cite{hartley1} goes through.
\end{proof}
In fact, every hyperbolic positive amphicheiral knot is \emph{periodically amphicheiral} by Mostow rigidity. This just means that the orientation reversing homeomorphism $h\colon S^3\to S^3$ which realizes the knot's amphicheirality is a finite order homeomorphism.

%\begin{definition}\label{defn10}
%\cite{Kaw2}A knot $K\subset \mathbb{R}^3$ is a periodic knot of period $n$ if there is a periodic map $f$ of $(\mathbb{R}^3,K)$ which satisfies following conditions.
%\begin{enumerate}
%\item $f$ is a $2\pi/n$-rotation about a line $F$ in $\mathbb{R}^3$
%\item $F\cap K=\emptyset$
%\end{enumerate}
%\end{definition} 

So the question is whether $C(z)\equiv f(z)f(-z)\mod 4$ for periodically positive amphicheiral knots. If yes, then we have solved Conjecture \ref{conj11} affirmatively.

\section {The Alexander polynomial and the Burau representation}\label{sect1}
Since we are focusing on periodically positive amphicheiral knots, we need a method to build these knots systematically. Let $w$ be a $(2n+1)$-braid, and $w^*$ be the mirror image of the braid $w$ (which means all crossings are reversed). Then the closure $ww^*$ is a strong positive amphicheiral knot and the closure $ww^*ww^*$ is a periodically positive amphicheiral knot. 
\begin{figure}[h!]
	\centering
	{\includegraphics[width=10cm]{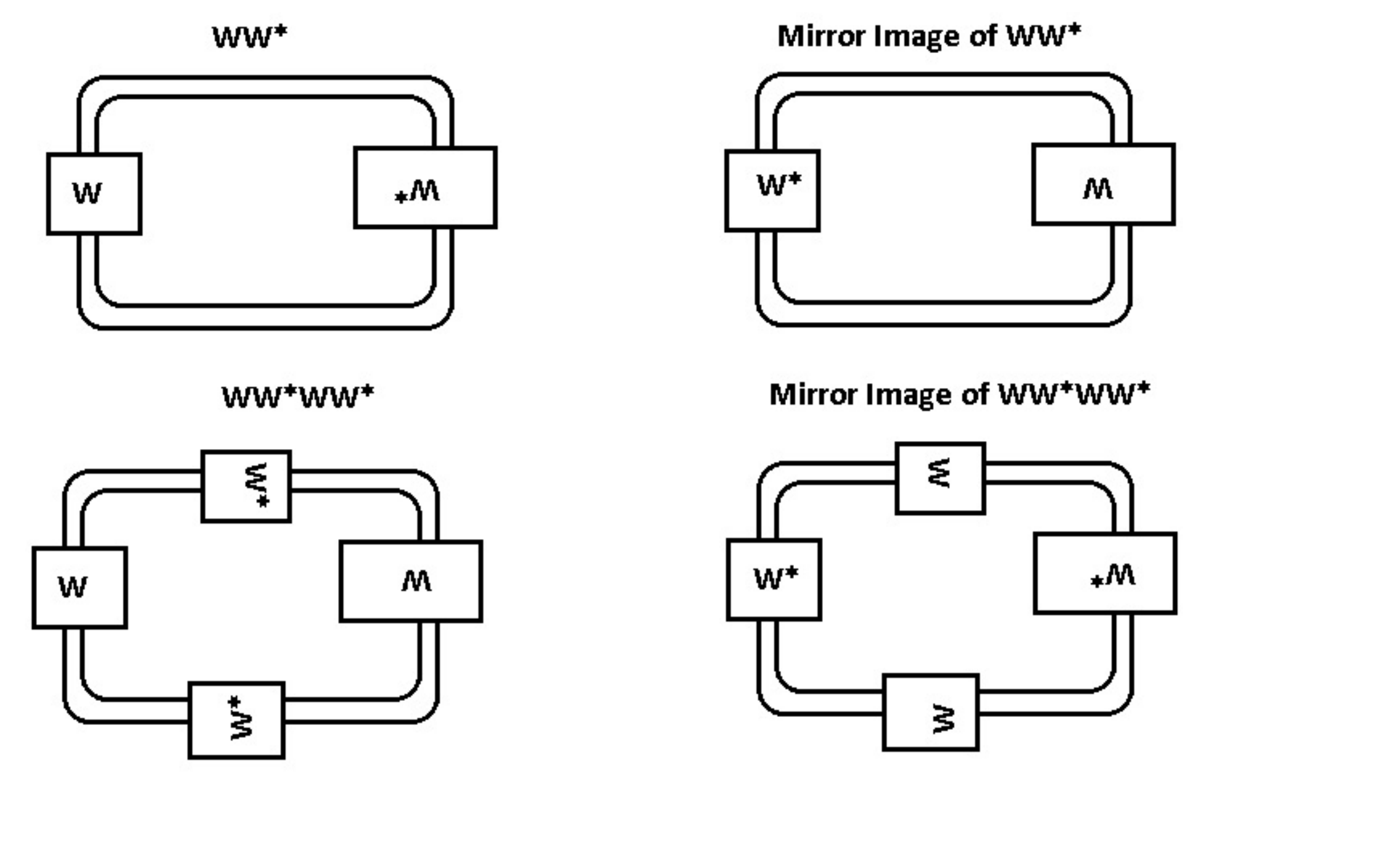}}
	  \caption{Braids $ww^*$ and $ww^*ww^*$ }\label{fig13}
\end{figure}

\begin{definition}
\cite{bir1} Let $\sigma_i$ denote the standard generator of the braid group $B_n$ with $n\ge 3$. Then the reduced Burau representation of $B_n$, for $ n \ge 3$ is given by
\begin{align*}
\sigma_1 &\mapsto \left( \begin{array}{cc|c}-t & 1 & 0 \\ 0 & 1 & 0 \\ \hline 0 & 0 & I_{n-3} \end{array} \right), \\
\sigma_i &\mapsto \left( \begin{array}{c|ccc|c} I_{i-2} & 0 & 0 & 0 & 0 \\ \hline 0 & 1 & 0 & 0 & 0 \\ 0 & t & -t & 1 & 0 \\ 0 & 0 & 0 & 1 & 0 \\ \hline 0 & 0 & 0 & 0 & I_{n-i-2} \end{array} \right), \quad 2 \leq i \leq n-2,  \\
\sigma_{n-1} &\mapsto \left( \begin{array}{c|cc} I_{n-3} & 0 & 0 \\ \hline 0 & 1 & 0 \\ 0 & t & -t \end{array} \right),
\end{align*}
Here $I_k$ denotes the  $k\times k$ identity matrix. For $n=2$ it maps $\sigma_1 \mapsto \left( -t \right)$.
\end{definition}
The following is well known. 
\begin{theorem}\label{theorem14}
If a knot $K$ is a closure of a braid $\Psi$ in $B_n$ then the Alexander polynomial of the knot $K$ is given by 
\[
\Delta_K(t)=\frac{\det(I-\beta(\Psi))}{1+t+t^2+...t^{n-1}} \text{ where } \beta(\Psi) \text{ is Burau representation of } \Psi
\]
\end{theorem}
Before we prove the main theorem of this section, we need a few results. Let $\Re$ be the operation on braids which reverses the order of a braid word. E.g. if $w=\sigma_1\sigma_3^2\sigma_2$ then $\Re(w)=\sigma_2\sigma_3^2\sigma_1$. If $L$ is a knot or link then let $\Re(L)$ be the knot or link with every strand orientation reversed. Let $w^*$ be the braid with all crossings reversed. Let $\Cl(w)$ denote the standard closure of a braid, and let $\hat{\Cl}(w)$ be the link $\Cl(w)\cup m$, where $m$ is an additional component representing the braid axis, oriented to link positively with $\Cl(w)$.

\begin{lemma}\label{lemma11}
For any braid $w$, $\hat{\Cl}(\Re(w))=\Re(\hat{\Cl}(w))$.
\end{lemma}
\begin{proof}
Reversing the strand orientation necessitates reading the braid in reverse order.
%Let $w= \sigma_{k_1}........\sigma_{k_n}$, where $k_i>0$ be a $m-$strand braid. After 180 rotation it  should be: $\sigma_{m-k_n}..........\sigma_{m-k_1}$ So when we take a reflection across the plane where $w$ is in, we get $\sigma_{-(m-k_n)}........\sigma_{-(m-k_1)}$. Now we take reflection through a plane perpendicular to the plane  where $w$ is in, $\sigma_{k_n}........\sigma_{k_1} =\Re(w)$. Here we denoted $\sigma_{i}^{-1}$ by $\sigma_{-i}$. This proves that closure of $w$ and $\Re(w)$ are ambient isotopic.
\end{proof}

\begin{lemma}\label{lemma12}
For any braid $w$,  $\Re((ww^*)^k)$ and $(ww^*)^{-k}$ are conjugate. 
\end{lemma}
\begin{proof}
Let $w= g_1\ldots g_n$, where $g_i$ are braid generators and their inverses. Then $ww^*=g_1\ldots g_ng^{-1}_{1}\ldots g_n^{-1})$.
Now $(ww^*)^{-1}=g_n\ldots g_1 g_n^{-1}\ldots g_1^{-1}$
 while $\Re(ww^*)=g_n^{-1}\ldots g_1^{-1}g_n\ldots g_1$. So the braids $\Re(ww^*)$ and $(ww^*)^{-1}$ represent the same cyclic words and are therefore conjugate.
To complete the proof, note that $\Re((ww^*)^k)=(\Re(ww^*))^k$ is conjugate to $((ww^*)^{-1})^k=(ww^*)^{-k}$.
 \end{proof}

\begin{lemma}\label{lemma13}
Let $A$ be an $n\times n$ matrix with $\det(A)=1$. Suppose $A$ and $A^{-1}$ have the same characteristic polynomial, $\chi_A$. Then the coefficient of $\lambda^k$ equals the coefficient of $\lambda^{n-k}$. That means $\chi_A=\lambda^n+a_1\lambda^{n-1}+a_2\lambda^{n-2}+...+a_2\lambda^{2}+a_1\lambda+1 $ 
\end{lemma}
\begin{proof}
When $A$ is nonsingular, it is easy to see that  $\chi_A(\lambda)=\chi_{A^{-1}}(\lambda^{-1})$ from which the result follows.
%
%\lambda^n+a_{n-1}\lambda^{n-1}+a_{n-2}\lambda^{n-2}+...+a_2\lambda^2+a_1\lambda+a_0$ be the characteristic polynomial of $A$. Then $\chi_{A^{-1}}(\lambda)=\lambda^n+\frac{1}{a_0}(a_1\lambda^{n-1}+a_2\lambda^{n-2}+...+a_{n-2}\lambda^2+a_{n-1}\lambda+1)$. However, $\det(A)=1\to a_0=1$. Since $\chi_A=\chi_{A^{-1}}$ we get the above result. 
\end{proof}

Now we give the definition of Lucas polynomials, which we will use in the proof of Theorem \ref{theorem15}. They are useful for converting between Alexander and Conway polynomials. (Recall that the Conway polynomial $C(z)$ and Alexander polynomial $A(t)$ are related by the change of variables $z=\sqrt{t}-\frac{1}{\sqrt{t}}.$)
\begin{definition}
The Lucas polynomials are defined recursively by  $L_0=2,L_1=z,L_n=zL_{n-1}+L_{n-2}$ for all $n\geq 2$.
\end{definition}
\begin{lemma}\label{LucasLemma} \cite{Agle}
Let $z=\sqrt{t}-\frac{1}{\sqrt{t}}$. Then $t^{n/2}+(-1)^nt^{-n/2}=L_{n}(z)$. 
\end{lemma}
\begin{proof}
The base cases of $k=0,1$ are obvious. So it suffices to show that  $t^{k/2}+t^{-k/2}$ satisfies the defining relation of the Lucas polynomials:
\begin{multline*}
 t^{n/2}+(-1)^nt^{-n/2}=\\ (t^{1/2}-t^{-1/2})(t^{(n-1)/2}+(-1)^{n-1}t^{-(n-1)/2})+(t^{(n-2)/2}+(-1)^{n-2}t^{-(n-2)/2}),
 \end{multline*}
which is not difficult.
\end{proof}

\begin{lemma}\label{LucasFac}
The Lucas polynomials satisfy the following identity in $\mathbb Z_4[z]$.
$$(1+L_1+\cdots+L_n)(1-L_1+\cdots+(-1)^nL_n))=(-1)^n(1+L_2+L_4+\cdots+L_{2n}).$$
\end{lemma}
\begin{proof}
We use the following identity \cite{Wu}, $L_nL_m=L_{n+m}+(-1)^mL_{n-m}$. We use induction on $n$. The base case is easy, so we consider the induction step. Assume first that $n$ is even. Then 
\begin{multline*}
((1+L_1+\cdots+L_{n-1})+L_n)((1-L_1+\cdots-L_{n-1})+L_n)=\\
 -(1+L_2+L_4+\cdots +L_{2n-2})+2L_n(1+L_2+L_4+\cdots+L_{n-2})+L_n^2\\
\end{multline*}
We have (in $\Z_4[z]$) 
\begin{multline*}
2L_n(1+L_2+L_4+\cdots+L_{n-2})=2(L_n+L_{n\pm 2}+\cdots+L_{n\pm (n-2)}=\\
2(L_2+L_4+\cdots+L_{2n-2})
\end{multline*}
So continuing the calculation, we have
\begin{multline*}
-(1+L_2+L_4+\cdots +L_{2n-2})+2(L_2+L_4+\cdots+L_{2n-2})+L_{2n}+2=\\
1+L_2+L_4+\cdots+L_{2n}
\end{multline*}
completing the inductive step when $n$ is even.
The inductive step when $n$ is odd is similar and left to the reader.
\end{proof}
\begin{theorem}\label{theorem15}
If $K$ is a positive amphicheiral knot coming from the closure of the braid $(ww^*)^{2^m}$, where $w$ is a braid with an odd number of strands, then the Conway polynomial of $K$ satisfies Conjecture \ref{conj11}. i.e\\
\[
C_K(z)\equiv f(z)f(-z)\mod 4 \text{ for some } f\in \mathbb{Z}_4[z]
\] 
\end{theorem}

\begin{proof}
We proceed by induction. We will use the fact that an $n$-braid closes to a knot iff any power of it closes to a knot. (The associated permutation must be an $n$-cycle.) Thus the induction will go smoothly.

For $m=1$, the resulting knot is strongly positive amphicheiral, so the result holds. Indeed $C_K(z)=f(z^2)^2$ for a polynomial in $\mathbb Z[z]$.

Consider the characteristic polynomial of $(ww^*)^{2^m}$ : $\det(\lambda I-\beta((ww*)^m))$ where $w$ is a $2n+1$-braid.
%$w^*$ is $w$ with all crossings reversed and $\beta(ww^*)$ represents the Burau representation of the braid $ww^*$. 
First we claim that the characteristic polynomials of $(ww^*)^k$ and $(ww^*)^{-k}$ are the same.
 This follows because the characteristic polynomial is the $2$-variable Alexander polynomial of the link which is the braid closure union the braid axis \cite{Morton}. In particular, Lemma \ref{lemma12} implies that $\Re(ww^*)$ and $(ww^*)^{-1}$ are conjugate, from which it follows that the link closures including the braid axes are the same up to a global orientation reversal (\ref{lemma11}). But the two variable Alexander polynomial is invariant under global orientation reversal. 

%This is true from Lemma \ref{lemma11} and \ref{lemma12} and the observation that $\det(\beta(ww^*))=1$. 
Hence, the characteristic polynomial $\chi_{\beta(ww^*)^k}(\lambda)$ is symmetric. When $\lambda=1$,  $\chi_{\beta(ww^*)^k}(1)=2(1+a_1(t)+\cdots+a_n(t))$. Similarly, $\chi_{\beta((ww^*)^k)}(-1)=2(1-a_1(t)+a_2(t)+\cdots+(-1)^na_n(t))$. Thus 
\begin{multline*}
\chi_{\beta(ww^*)^k}(1)+(-1)^n\chi_{\beta(ww^*)^k}(-1)=\\\det(I-\beta{(ww^*)^k})+(-1)^n\det(I+\beta{(ww^*)^k})\equiv 0\mod 4.
\end{multline*}
 In this proof we have used fact that $a_i(t)$, when converted using the substitution $\sqrt{t}-\frac{1}{\sqrt{t}}=z$ yields integer polynomials.
 
 Now observe $(1+t+\cdot\cdot\cdot t^{2n})$ differs only by a power of $t$ from $((t^n+\frac{1}{t^n})+\cdot\cdot\cdot+(t+1/t)+1)=L_{2n}(z)+L_{2n-2}(z)+\cdots+1$ by Lemma~\ref{LucasLemma}. Now, by Lemma~\ref{LucasFac}, this is equivalent to $(-1)^n(1+L_1(z)+\cdot\cdot\cdot+L_n(z))(1-L_1(z)+\cdot\cdot\cdot+(-1)^nL_n(z))) \mod 4$.
 %Here $z=\sqrt{t}-1/\sqrt{t}$ and $L_n(z)$ is the $n^{th}$ Lucas Polynomial.
 
The induction hypothesis is that the Alexander polynomial of the knot closure of $(ww^*)^{2^{m-1}}$ is \[
 \frac{\det(I-\beta((ww^*)^{2^{m-1}}))}{1+t+\cdot\cdot\cdot+t^{2n}}\equiv f(z)f(-z) \text{ for some} f(z)\in \Z_4[z]
 \]. 
 
From above we have, $\det(I-\beta{(ww^*)^{2^{m-1}}})+(-1)^n\det(I+\beta{(ww^*)^{2^{m-1}}})\equiv 0\mod 4$. Thus we conclude 
\[
\det(I+\beta(ww^*))\equiv f^2(z)(1+L_1(z)+\cdot\cdot\cdot+L_n(z))(1-L_1(z)+\cdot\cdot\cdot+(-1)^nL_n(z))) \mod 4
\]
 Now consider the closure of braid $(ww^*)^{2^m}$. We can see that  
 
 \begin{align*}
 \frac{\det(I-\beta(ww^*)^{2^m})}{1+t+\cdot\cdot\cdot+t^{2n}}&=\frac{\det(I-\beta(ww^*)^{2^{m-1}})\det(I+\beta(ww^*)^{2^{m-1}})}{1+t+\cdot\cdot\cdot+t^{2n}}\\
&\equiv f^2(z)f^2(-z)(1+L_1(z)+\cdot\cdot\cdot+L_n(z))(1-L_1(z)+\cdot\cdot\cdot+(-1)^nL_n(z))\\
 &\equiv g(z)g(-z) \mod 4 \\
  \text{ where } g(z) &=f^2(z)(1+L_1(z)+\cdot\cdot\cdot+L_n(z))
 \end{align*}
 Here we use the easily proven fact that $L_{2n+1}(z)$ contains only odd powers of $z$. 
 
 This concludes the proof of the inductive step.
% For $(ww^*)^{2^m}, m\ge 3$ observe that the above method is still valid, and we have following induction hypothesis :
% $\frac{\det(I-\beta((ww^*)^{2^{m-1}}))}{1+t+\cdot\cdot\cdot+t^{2n}}\equiv g(z)g(-z) \mod 4$ for some $g\in \mathbb{Z}[z]$. [this needs rephrasing!] 
 
\end{proof}

\begin{corollary}
If $K$ is a positive amphicheiral knot coming from the closure of the braid $(ww^*)^{k}$, where $w$ is a braid with an odd number of strands, then the Conway polynomial of $K$ satisfies Conjecture \ref{conj11}. i.e\\
\[
C_K(z)\equiv f(z)f(-z)\mod 4 \text{ for some } f\in \mathbb{Z}_4[z]
\] 
\end{corollary}
\begin{proof}
Suppose $k=2q+1$ is odd. Then $(ww^*)^{k}=\tilde{w}\tilde{w}^*$ where $\tilde{w}=(ww^*)^qw$. In general we have $k=2^m(2q+1)$, and 
$$(ww^*)^k=(w^{2q+1})^{2^m}=(\tilde{w}\tilde{w}^*)^{2^m},$$
which is of the form where Theorem~\ref{theorem15} applies.
\end{proof}

%
% Thus Theorem \ref{theorem15}, essentially proves the conjecture for all amphicheiral knots coming from braid closure $(ww^*)^{2^m.k}$, with $m,k\in\Z$.

We close this section with an example where we calculate the Conway Polynomials explicitly for an infinite class of $3$ strand braids of the form $ww^*ww^*$. In view of Theorem~\ref{theorem12}, these all have Conway polynomials which split over $\mathbb Z_4[z]$, but in general they will not split over $\mathbb Z[z]$. These examples are closely related to Fibonacci polynomials, which we now introduce.

\begin{definition}\cite{hoggatt}
The two variable generalized Fibonacci polynomial is defined recursively as follows:\\
$U_n(x,y)=xU_{n-1}(x,y)+yU_{n-2}(x,y)$, $U_1(x,y)=1,U_0(x,y)=0$ \\
The one variable Fibonacci polynomial is defined  by,\\ 
 $F_1(z)=1,F_2(z)=z,F_n(z)=zF_{n-1}(z)+F_{n-2}(z)$. 

\end{definition}

\begin{corollary}\label{corol11}
Let $\Omega_n(t)=\Sigma_{i=0}^{n-1}(-t)^i$. When $n$ is odd, $\frac{\Omega_n^2}{t^{n-1}}=F_n^2(z)$, where $F_n$ denotes the $n^{th}$ Fibonacci polynomial and $z=\sqrt{t}-\frac{1}{\sqrt{t}}$.
\end{corollary}
\begin{proof}
We need to prove $t^{-(n-1)/2}\Omega_n(t)=F_n(t)$, when $n$ is odd. First we claim that $U_n(t-1,t)=(-1)^{n-1}\Sigma_{i=0}^{n-1}(-t)^i$. This can be proved using induction. In particular when $n$ is odd, $U_n(t-1,t)=\Omega_n(t)$. Next we use following identity \cite{hoggatt}. $U_n(x,y)=y^{(n-1)/2}F_n(\frac{x}{\sqrt{y}})$. So we have $U_n(t-1,t)=t^{(n-1)/2}F_n(\sqrt{t}-\frac{1}{\sqrt{t}})=t^{(n-1)/2}F_n(z)$. Letting $n$ be odd, we get our result.\\ 
\end{proof}

\begin{proposition}\label{lemma14}
Let $K$ be the closure of the 3-braid $ww^*ww^*$, where $w=\sigma_1^n\sigma_2^m, w^*=\sigma_1^{-n}\sigma_2^{-m}$ with $ n,m $ are odd integers. Then  
\begin{align*}
C_K(z) &=F_n^2(z)F_m^2(z)(4-F_n^2(z)F_m^2(z)(z^2+3)) 
\end{align*} 
 where $F_n(z)$ is n-th Fibonacci polynomial. \end{proposition}
 \begin{proof}
Consider the Alexander polynomial, $\Delta_K(t)$, of the closure of $ww^*ww^*$. Then using the Burau representation,
\begin{align*}
 \Delta_K(t)&=\frac{\det(I-\beta(ww^*ww^*))}{1+t+t^2}\\
 			&=\frac{\det(I-\beta(ww^*))\det(I+\beta(ww^*))}{1+t+t^2}
 \end{align*}
 Let matrices $A,B$ denote the Burau representation of $\sigma_1$ and $\sigma_2$ respectively. That means,
 \[
 A= \left( \begin{array}{cc}-t & 1  \\ 0 & 1 \end{array} \right), \quad B= \left( \begin{array}{cc}1 & 0  \\ t & -t \end{array} \right)
\]
Since $w=\sigma_1^n\sigma_2^m$, this implies,
\begin{align*}
 \Delta_K(t)&=\frac{\det(I-A^nB^mA^{-n}B^{-m}A^nB^mA^{-n}B^{-m})}{1+t+t^2}\\
 			&=\frac{\det(B^mA^n-A^nB^m)\det(B^mA^n+A^nB^m)}{(1+t+t^2)\det(A^nB^m)^2}
 \end{align*}
 Now observe that,
  \begin{align*}
  A^n &= \left( \begin{array}{cc}(-t)^n & \Omega_n(t)  \\ 0 & 1 \end{array} \right), \quad B^m= \left( \begin{array}{cc}1 & 0  \\ t\Omega_m(t) & (-t)^m \end{array} \right)\\
  A^nB^m &=\left( \begin{array}{cc}(-t)^n+t\Omega_m\Omega_n & (-t)^m\Omega_n  \\ t\Omega_m & (-t)^m \end{array} \right), \quad B^mA^n= \left( \begin{array}{cc}(-t)^n & \Omega_n  \\ t(-t)^n\Omega_m & t\Omega_m\Omega_n \end{array} \right)\\
 \end{align*}
 here $\Omega_n=\Sigma_{i=0}^{n-1}(-t)^i$. Also note that $\det(A^nB^m)=t^{n+m}$. Now, tedious calculations shows that,
 \begin{align*}
 \det(B^mA^n-A^nB^m)&=\Omega_n\Omega_m t(1+t^m+t^n+t^{m+n}-\Omega_m\Omega_nt) \text{ and}\\
 \det(B^mA^n+A^nB^m)&=4(-t)^{m+n}-\Omega_n\Omega_m t(1+t^m+t^n+t^{m+n}-\Omega_m\Omega_nt)
 \end{align*} 
 
 Now using the identity, $\frac{1+t^m+t^n+t^{m+n}-\Omega_m\Omega_nt}{1+t+t^2}=\Omega_m\Omega_n$,\\
\[
 \Delta_K(t)=\frac{\Omega_m\Omega_n t(4t^{m+n}-\Omega_m^2\Omega_n^2t(1+t+t^2))}{t^{2(m+n)}}
\]
Now using the Corollary \ref{corol11}, we have $\frac{\Omega_n^2}{t^{n-1}}=F_n^2(z)$ and $\frac{\Omega_m^2}{t^{m-1}}=F_n^2(z)$. Substituting these identities in above equation, we get
\[
\Delta_K(t)=\frac{F_n^2F_m^2}{t}(4-F_n^2F_m^2(t+\frac{1}{t}+1))
\]
 Now observing $t+\frac{1}{t}+1=z^2+3$ and the Alexander polynomial is defined up to multiplication by $t^{\pm 1}$, we conclude that the Conway polynomial of the closure of braid $ww^*ww^*$ is given by $C(z)=F_n^2F_m^2(4-F_n^2F_m^2(z^2+3))$. 
 \end{proof}
 
% In above $n,m$ need to be odd, in order to get a knot from the closure of $ww^*ww^*$. 
% Otherwise, we would get a link with more than one component. The beauty of the Lemma \ref{lemma14} is, it enable us to write the Conway polynomial of amphicheiral knot, $\sigma_1^n\sigma_2^m\sigma_1^{-n}\sigma_2^{-m}\sigma_1^n\sigma_2^m\sigma_1^{-n}\sigma_2^{-m}$ without using the definition of the Conway polynomial(avoiding the process of resolving crossings). The only parameters we need are $n$ and $m$. This theorem suggests that there may be a general relationship between the Conway polynomial of amphicheiral knots and the Fibonacci polynomial. Also note that under $\mod4$, 
Note that
 \[
 C_K(z)\equiv F_n^4F_m^4(z-1)(z+1)\mod 4
 \]
consistent with Theorem~\ref{theorem15}. In fact, in the next proposition we give two special cases of the preceding proposition where $C_K(z)$ splits over the integers: when  $n=m$ and when one of $n,m$ is equal to $1$. However, in general $C_K(z)$ does not split. Indeed $\varphi_{n,m}=(4-F_n^2(z)F_m^2(z)(z^2+3))$ is irreducible in all other cases where $n,m$ are odd integers less than $100$, and it may be that it never splits except in the two cases of Proposition~\ref{factor}.
  
 \begin{proposition}\label{factor}
 Let $m$ be odd. 
 \begin{enumerate}
 \item Then $\varphi_{1,m}=p(z)p(-z)\in\mathbb Z[z],$ where $p(z)=F_{m+1}+F_m+F_{m-1}$.
 \item $\varphi_{m,m}=p(z)p(-z)\in\mathbb Z[z]$, where $p(z)=
F_m^2+F_{2m}-2$.
 \end{enumerate}
 \end{proposition}
 \begin{proof}
 The proof of part (a) uses the fact that the Fibonacci polynomials $F_m$ are a basis for $\mathbb Z[z]$ together with the multiplication rule
 $$F_nF_m=F_{n+m-1}-F_{n+m-3}+\cdots \pm F_{|n-m|+1}.$$
 
 Part (b) follows from the identity $$3F_n^4+F_3F_n^4+4(-1)^nF_n^2-F_{2n}^2=0.$$
 This can be proven using the identity
 $$F_n=\frac{\alpha^n-\bar\alpha^n}{\alpha-\bar\alpha},$$
 where $\alpha=\frac{z+\sqrt{z^2+4}}{2}$ and $\bar\alpha=\frac{z-\sqrt{z^2+4}}{2}$.
 \end{proof}
 
 \section{The leading coefficient of the Conway polynomial and amphicheiral knots }
 In this section we explore two conjectures proposed by A. Stoimenow regarding the leading coefficient of the Conway polynomial of an amphicheiral knot. The conjectures are based on the following theorem.
 \begin{theorem}\label{theorem16}
 \cite{St1}Let $K$ be an amphicheiral knot. Then the leading coefficient of the Alexander( and Conway) polynomial of the knot $K$,  $\Delta_K(z)$ is a square(up to sign) and the sign of the leading coefficient of the Alexander polynomial $\Delta_K(z)$ is $(-1)^{\textbf{max deg } \Delta_K(z)}$ if any of the following hold.
 \begin{enumerate}
 \item $K$ is an alternating knot,
 \item $K$ is strongly amphicheiral or negative amphicheiral,
 \item $K$ is any knot with at most 16 crossings or
 \item $K$ is a fibered homogeneous knot. 
 \end{enumerate}
 \end{theorem}
 This theorem provides strong background for following conjectures. 
 \begin{conjecture}\label{conj13}
 	\cite{St1}The leading coefficient of the Alexander (or Conway) polynomial of an amphicheiral knot is a square.
 \end{conjecture}
 
 \begin{conjecture}\label{conj14}
 	\cite{St1}The sign of the leading coefficient of the Alexander polynomial, $\Delta_K$ of an amphicheiral knot $K$ is $(-1)^{\textbf{max deg} \Delta_K}$.
 \end{conjecture}
 
 It is easy to see that any counterexample to Conjecture \ref{conj13} is again a counterexample to the Kawauchi conjecture \ref{conj12}. More formally we can say,
 \begin{fact}
 If the leading coefficient of the polynomial $G(z)\in \mathbb{Z}[z]$ is not square then $G(z)\ne F(z)F(-z)$ for any $F(z)\in \mathbb{Z}[z]$.
 \end{fact}
 
 Now we give several counterexamples to conjectures we mentioned at the beginning.
 \begin{counter}\label{counter11}
 Let $w=\sigma_1^2\sigma_2\sigma_1^{-1}\sigma_2\sigma_1\sigma_3^{-1}\sigma_2\sigma_4\sigma_3^{-1}\sigma_4^2$, which is a 5-strand braid. Then the Conway polynomial of the closure of a $ww^*ww^*$ is $1 - 3 z^2 - 33 z^4 + 54 z^6 + 535 z^8 + 869 z^{10} + 489 z^{12} + 72 z^{14}$. The knot $ww^*ww^*$ is shown in the Figure \ref{fig22}. It was obtained using Knotscape\cite{knotscape}.
 \end{counter}
 
 \begin{figure}[h!]
\centering
\includegraphics[width=5cm]{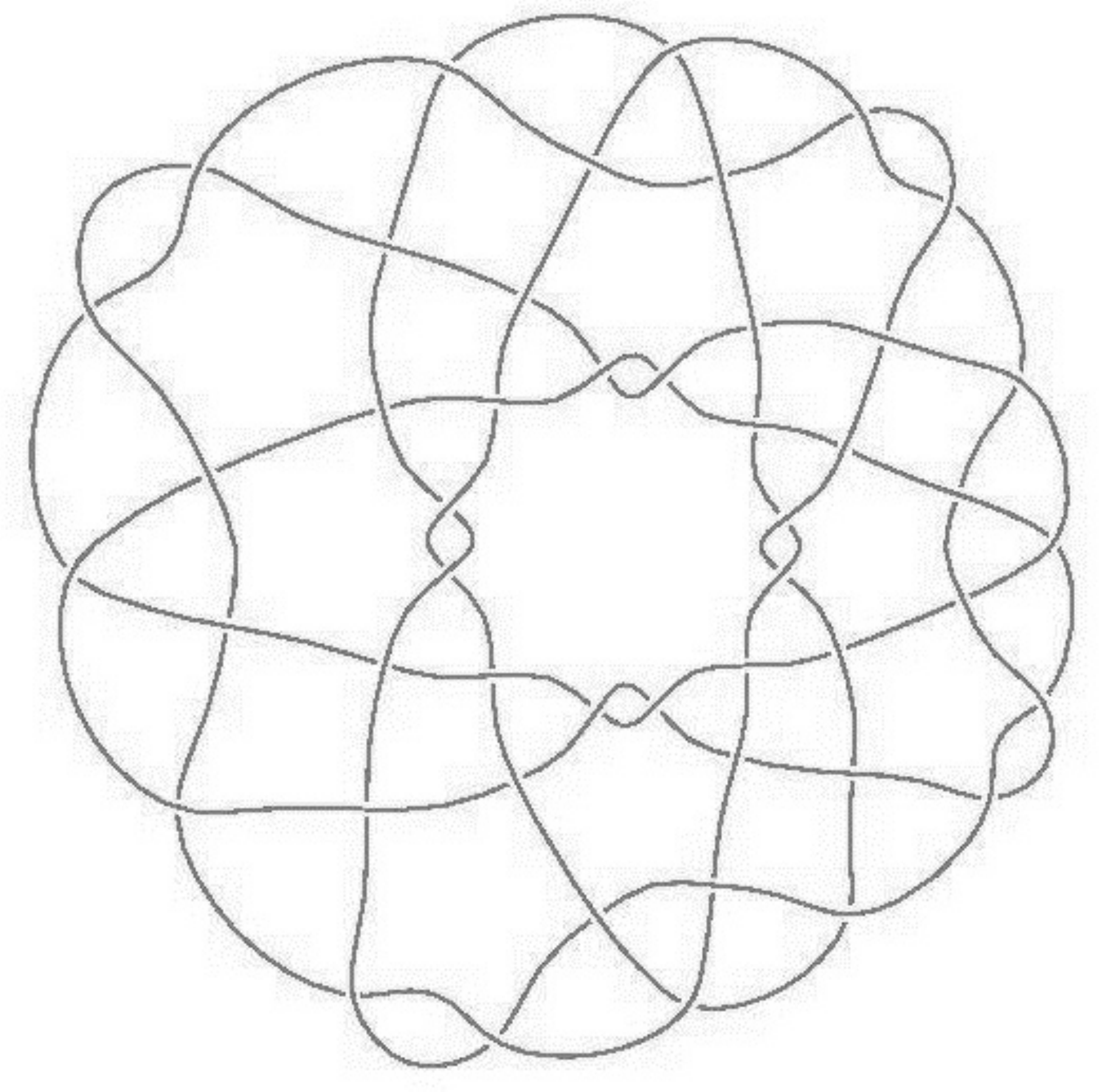}
 	  \caption{Amphicheiral knot with the leading coefficient 72 }\label{fig22}
 \end{figure}

Theorem 6 in \cite{Mu} states that if $K$ is an alternating knot and the leading coefficient of the Alexander polynomial is prime, then $K$ is non-amphicheiral. With this theorem, one may think that the leading coefficient of the Alexander polynomial of an amphicheiral knot which does not satisfy above conjectures must be composite. However the leading coefficient of the Alexander polynomial of the following amphicheiral knot is prime and serves as a counterexample to three Conjectures \ref{conj12}, \ref{conj13} and \ref{conj14}. 

\begin{counter}\label{counter12}
Let $w=\sigma_1^2\sigma_2\sigma_1^{-1}\sigma_3^{-1}\sigma_2^{-1}\sigma_1\sigma_4\sigma_3\sigma_2^2\sigma_3^2\sigma_4$ be the 5-strand braid. Then the Conway polynomial of the closure of $ww^*ww^*$ is $1 + 5 z^2 + 39 z^4 + 246 z^6 + 657 z^8 + 743 z^{10} + 301 z^{12} -78 z^{14} - 105 z^{16} - 31 z^{18} - 3 z^{20}$. The knot is positive amphicheiral but not strongly amphicheiral.
\end{counter}

After seeing these counterexamples, one may wish to see a prime non-alternating knot which does not satisfy all of the above three conjectures. Upon our request, M. Thistlethwaite \cite{Mor} kindly provided us the list of 19-crossings prime positive amphicheiral knots. In that list, the following knot does not satisfy any of the above three conjectures and it is a non-alternating positive(not strong) amphicheiral knot. 

\begin{counter}\label{counter13}
Let $K$ be the knot with the following Dowker-Thistlethwaite code. 
\[
 6, -12, 32, -18, -26, 16, -4, -22, 34, -38, 30, -14, 20, 36, -10, 24, 2 , 28, -8
\]

Then $K$ is a prime non-alternating and positive non strong amphichiral knot with the Conway polynomial, $1+3z^2+8z^4$. The knot is depicted in the Figure \ref{fig33}. It was obtained using Knotscape\cite{knotscape}.Observe that $1+3z^2+8z^4\equiv(1-z)(1+z)\mod 4$. Thus it is not a counterexample to Conjecture \ref{conj11}

\end{counter}

 \begin{figure}[h!]
\centering
\includegraphics[width=5cm]{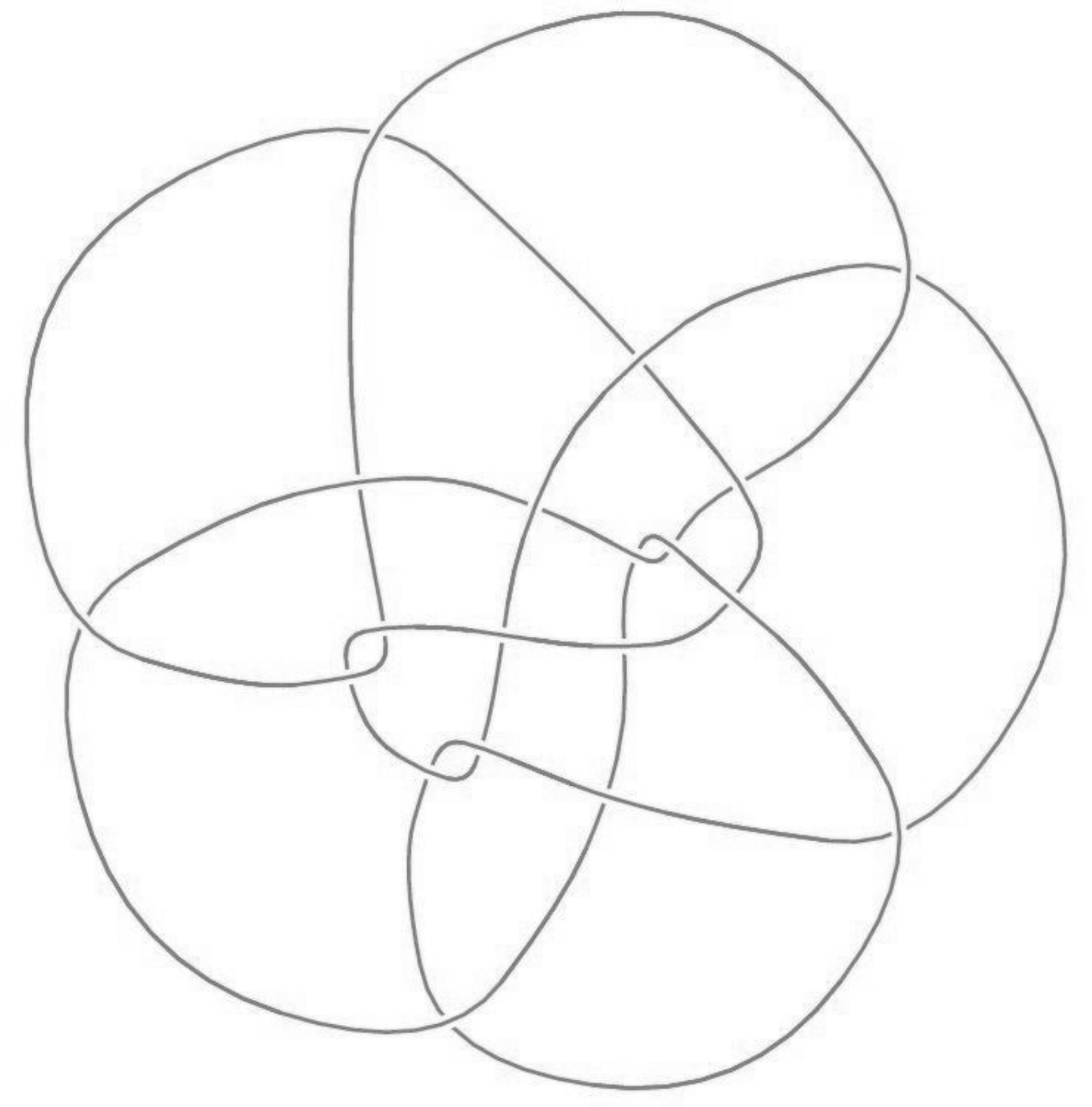}
 	  \caption{Non alternating prime amphicheiral knot with the leading coefficient 8 }\label{fig33}
 \end{figure}
We finish this section with following realization problem. 
\begin{question}
Let $n$ be an arbitrary integer. Is it possible to have an amphicheiral knot $K$ with $n$ as a leading coefficient of its Conway polynomial? 
\end{question}
Conjecture ~\ref{conj11} puts restrictions on the $n$ that can appear. Namely, if $\deg(C_K(z))=2n$, then the leading coefficient must be of the form $(-1)^n(4k+1)$ or divisible by $4$. In particular no number of the form $4k+2$ can appear as the leading coefficient of an amphicheiral knot if Conjecture~\ref{conj11} is true.

If $n$ is a square, then we have a positive answer to the preceding question due to E. Flapan. Namely she has shown if $C(z)=F(z)^2$ and $F(0)=1$ for some integer polynomial $F(z)$, then there exists a strong positive amphicheiral knot which has $C(z)$ as its Conway polynomial \cite{Flap}. The question will be more interesting when the absolute value of $n$ is prime. 

\begin{question}\label{q1}
For a given prime $p$, is there an amphicheiral knot such that the absolute value of the leading coefficient of the Conway polynomial is $p$?
\end{question}

If the answer to the above question is yes, we can certainly choose the amphicheiral knot to be prime and even hyperbolic by the factorization formula for the Alexander Polynomial in terms of the knot complement's JSJ components cf. the proof of Theorem~\ref{theorem13}.  That means we have prime amphicheiral knots with prime numbers as leading coefficient of their Conway polynomials. It is clear from Theorem \ref{theorem16}, if it exists then they are non alternating positive non strong amphicheiral knots.

We found several amphicheiral knots which have odd prime numbers as leading coefficients. The table \ref{leadtable} and the Figure \ref{primek} summarize our finding. Figures were obtained using Knotscape\cite{knotscape} and the Conway polynomial calculated using the Mathematica package KnotTheory\cite{knottheory, wolfram}. 

  \begin{table}[h!]
  \begin{center}
\begin{footnotesize}
  \caption{Leading coefficients of the Conway polynomial for knot $ww^*ww^*$  }
  \label{leadtable}
\begin{tabular}{|c|c|c|} \hline
Braid word $w$ & \multicolumn{1}{|p{3cm}|}{\centering Leading coefficient \\ of the Conway polynomial \\ of $ww^*ww^*$} & \multicolumn{1}{|p{2cm}|}{\centering Number of \\strands} \\ \hline
$\sigma_1^2\sigma_2\sigma_1^{-1}\sigma_3^{-1}\sigma_2^{-1}\sigma_1\sigma_4\sigma_3\sigma_2^2\sigma_3^2\sigma_4$ & -3 & 5\\ \hline
$\sigma_1\sigma_2^{-1}\sigma_3\sigma_2^{-1}\sigma_1^{-1}\sigma_2^{-1}\sigma_4^{-1}\sigma_3\sigma_2^{-1}\sigma_3\sigma_4^{-1}\sigma_4^{-1}\sigma_4^{-1}\sigma_4^{-1}$ & 5 & 5\\ \hline
$\sigma_1\sigma_2\sigma_3^{-1}\sigma_4\sigma_4\sigma_4\sigma_3^{-1}\sigma_2^{-1}\sigma_1^{-1}\sigma_3^{-1}\sigma_2^{-1}\sigma_3^{-1}\sigma_4\sigma_3^{-1}\sigma_2\sigma_3^{-1}\sigma_4\sigma_3^{-1}$ & -7 & 5 \\ \hline
$\sigma_1\sigma_2^{-1}\sigma_3^{-1}\sigma_4\sigma_3^{-1}\sigma_2\sigma_1^{-1}\sigma_3^{-1}\sigma_3^{-1}\sigma_4^{-1}\sigma_3\sigma_2\sigma_2\sigma_4^{-1}\sigma_3^{-1}\sigma_2$ & -11 & 5\\ \hline
$\sigma_1\sigma_2\sigma_3\sigma_4\sigma_3\sigma_3\sigma_3\sigma_3\sigma_2^{-1}\sigma_1^{-1}\sigma_2\sigma_3^{-1}\sigma_3^{-1}\sigma_2\sigma_2\sigma_3\sigma_4^{-1}\sigma_3$ & 13 & 5\\ \hline
$\sigma_1\sigma_2^{-1}\sigma_2^{-1}\sigma_3\sigma_2^{-1}\sigma_4^{-1}\sigma_4^{-1}\sigma_3\sigma_2^{-1}\sigma_1^{-1}\sigma_2^{-1}\sigma_4^{-1}\sigma_3\sigma_2^{-1}$
& 17 & 5 \\\hline
\end{tabular} 
\end{footnotesize}
  \end{center}
 \end{table}
The table suggest that using 5- strand braid words $w$, we may be able to get all odd prime numbers (in absolute value) as a leading coefficient.  %The table also directs us to the question of whether we can have $\pm 2$ as the leading coefficient of positive non strong amphicheiral knots.  
Conjecture \ref{conj11} implies that $\pm2$ cannot appear as a leading coefficient.
%This is quite interesting because if it exists, then it serves as a counterexample to Conjecture \ref{conj11}. %(If $\pm 2$ is the leading coefficient of the Conway polynomial then $C(z)C(iz)-C(z^2)\not\cong 0\mod 4$.)

%For periodic knots we have the following result from \cite{davis}
%\begin{theorem}\cite{davis}
%If $K$ is of period $p^r$ and the leading coefficient of the Conway polynomial $\Delta_K$ is divisible by $p$, then it is divisible by $p^r$.
%\end{theorem}
%
%Now for periodically positive non strong amphicheiral knots, period is $2^r$ with $r\ge 4$.Thus for periodically positive non strong amphicheiral knots we conclude that the leading coefficient of the Conway polynomial cannot be 2. If there exist an amphicheiral knot with the leading coefficient of the Conway polynomial is 2, it must be a non periodic,non hyperbolic, non alternating, non strong positive amphicheiral knot. The non hyperbolic condition comes from the following proposition. 
%\begin{proposition}(\cite{Kaw2}, Proposition 10.4.3 )
%If $K$ is a hyperbolic amphicheiral knot then it is a periodic amphicheiral knot.
%\end{proposition} 
%It is known that there are amphicheiral knots which are not periodically amphicheiral\cite{hartley3}. We leave the question \ref{q1} as future research question.
\begin{figure}
\begin{scriptsize}
\begin{tabular}{ccc} 
	{\includegraphics[width=4cm]{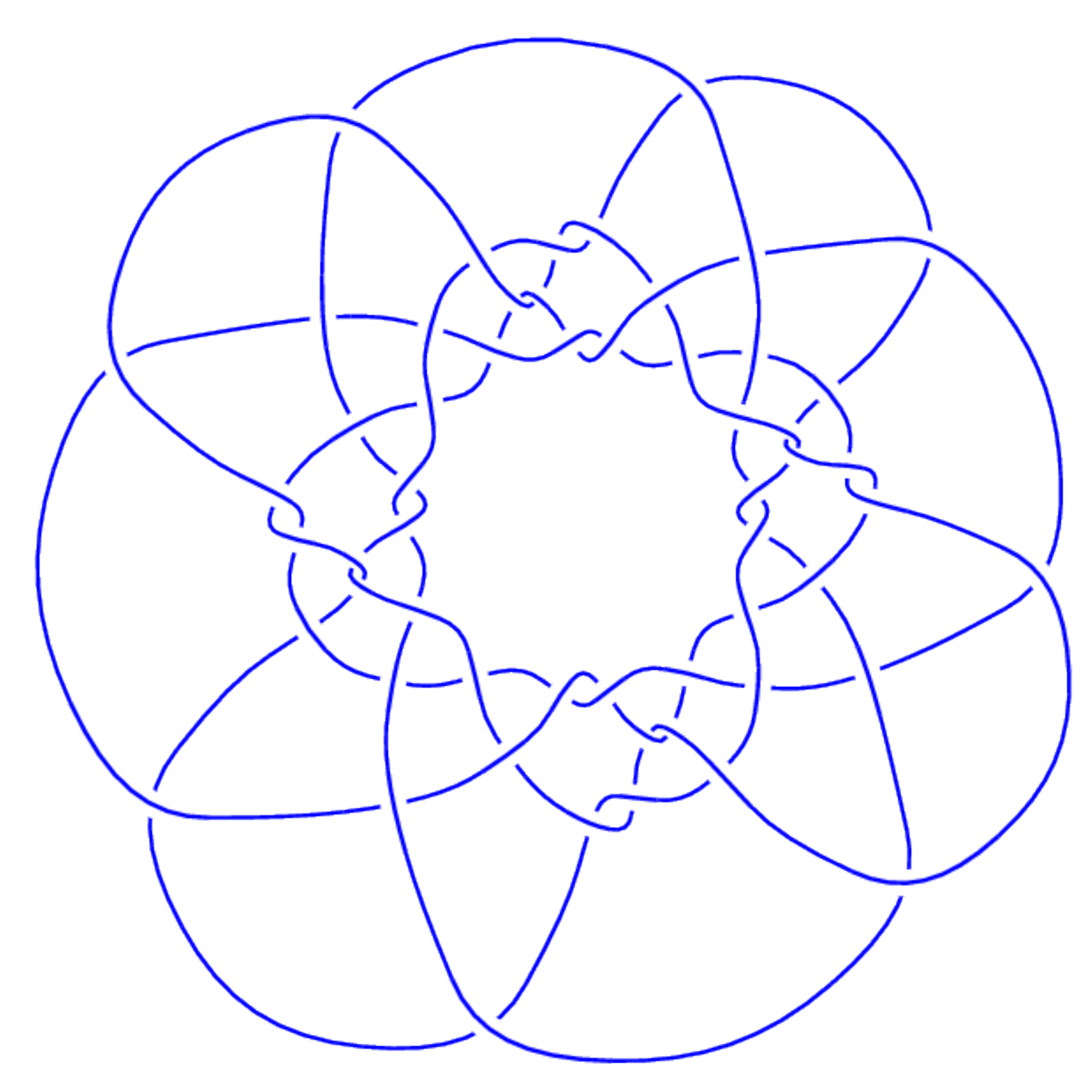}} &	\qquad &{\includegraphics[width=4cm]{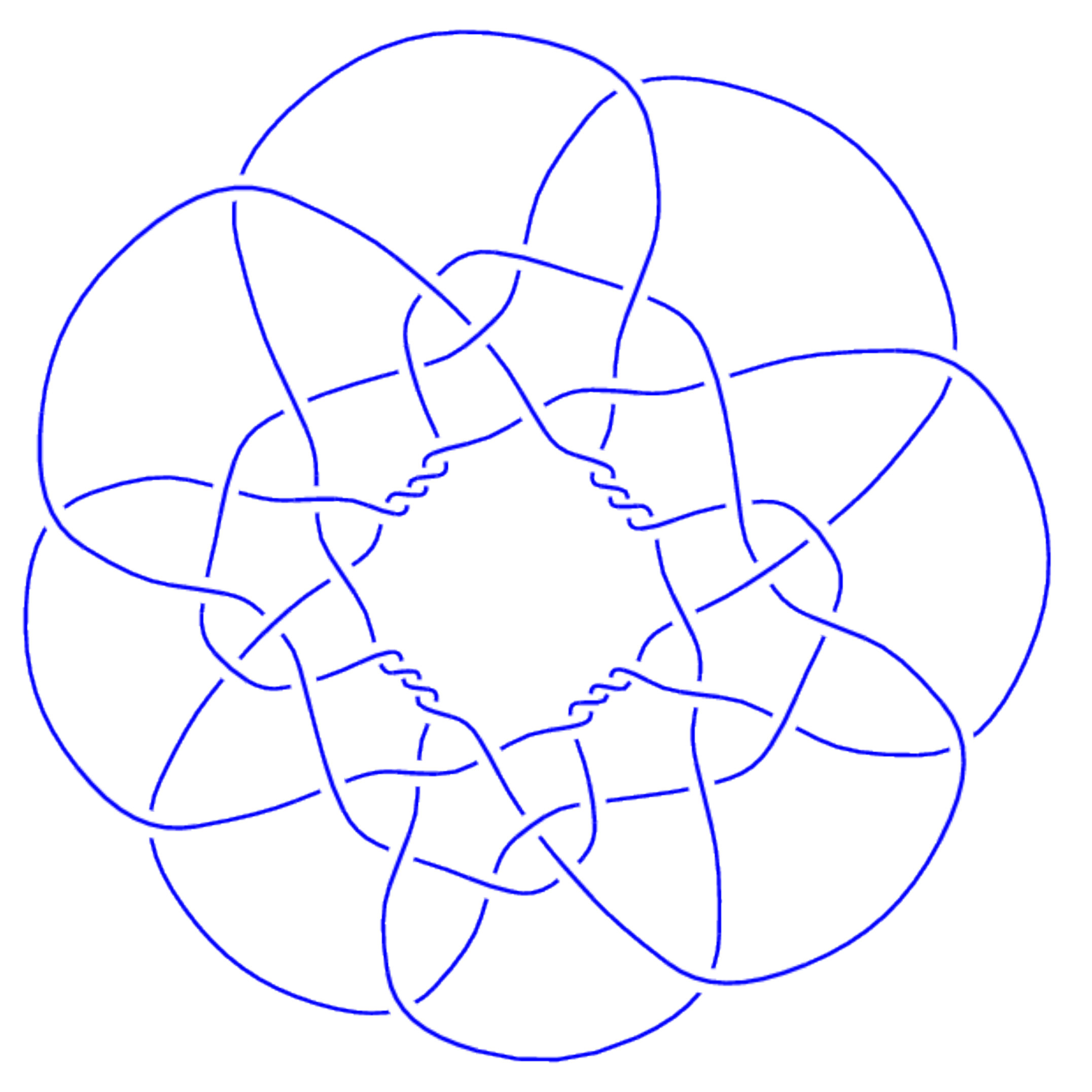}} \\
	Leading coefficient = -3 & \qquad &Leading coefficient = 5 \\
	{\includegraphics[width=4cm]{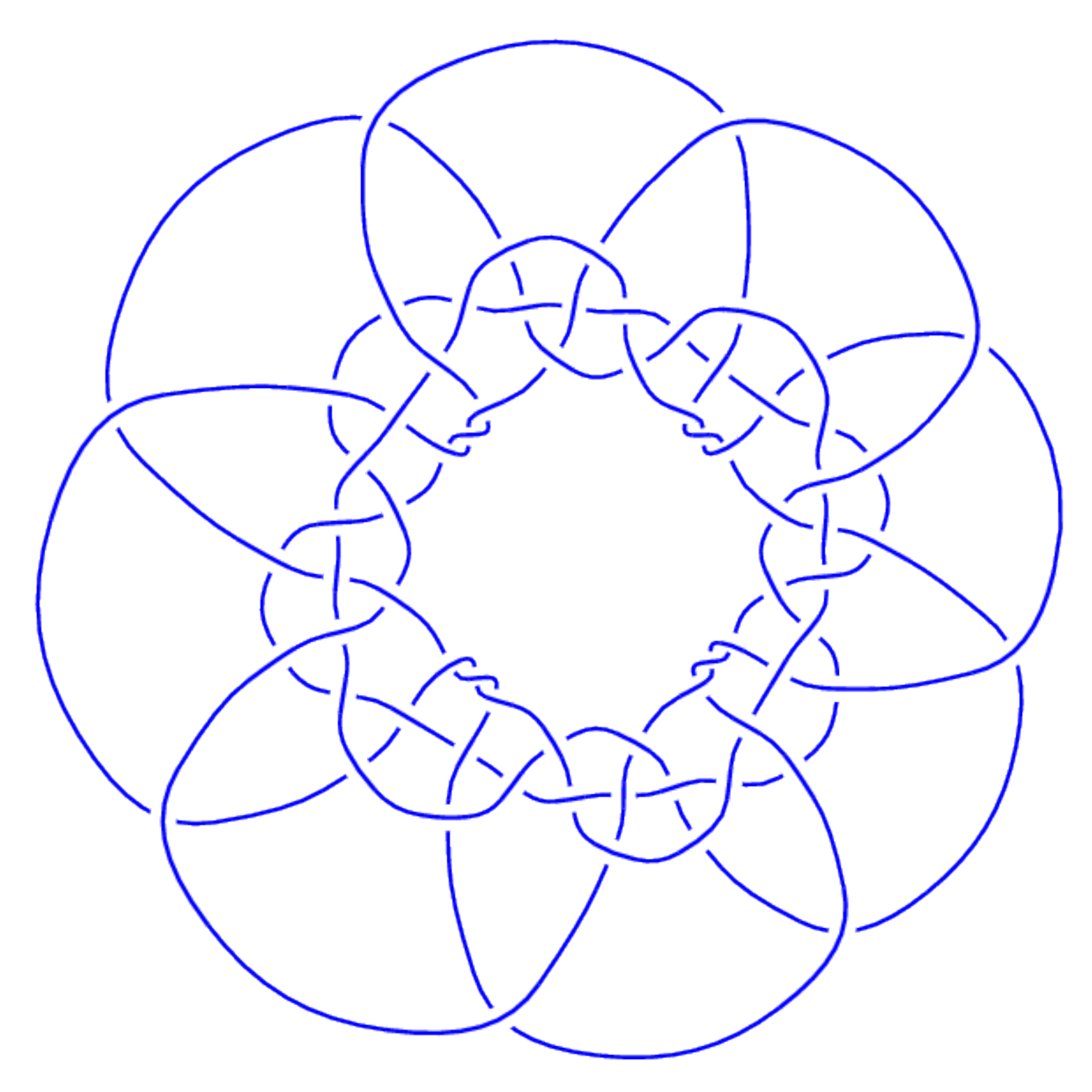}} & \qquad&	{\includegraphics[width=4cm]{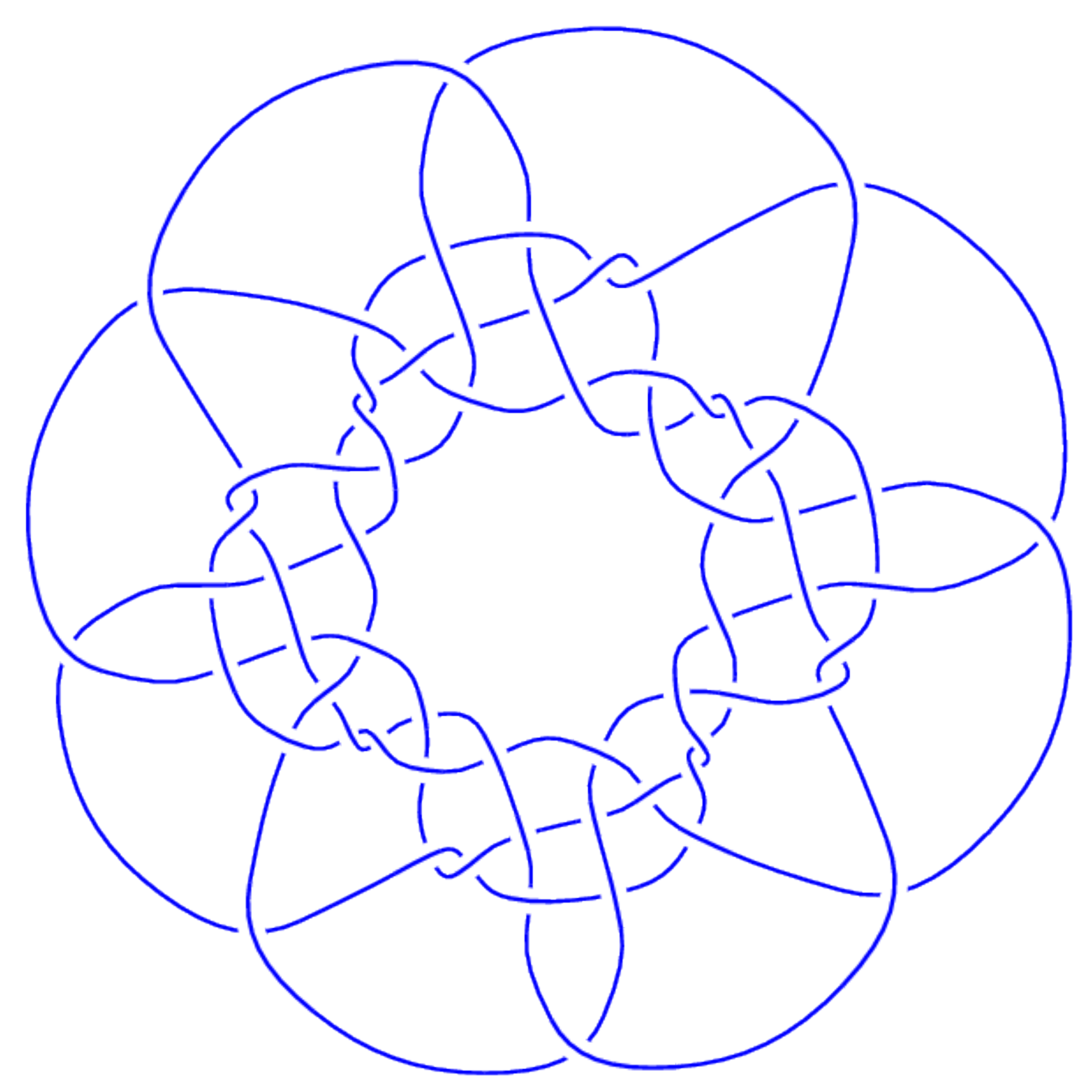}} \\
		Leading coefficient = -7 &\qquad& Leading coefficient = -11 \\
			{\includegraphics[width=4cm]{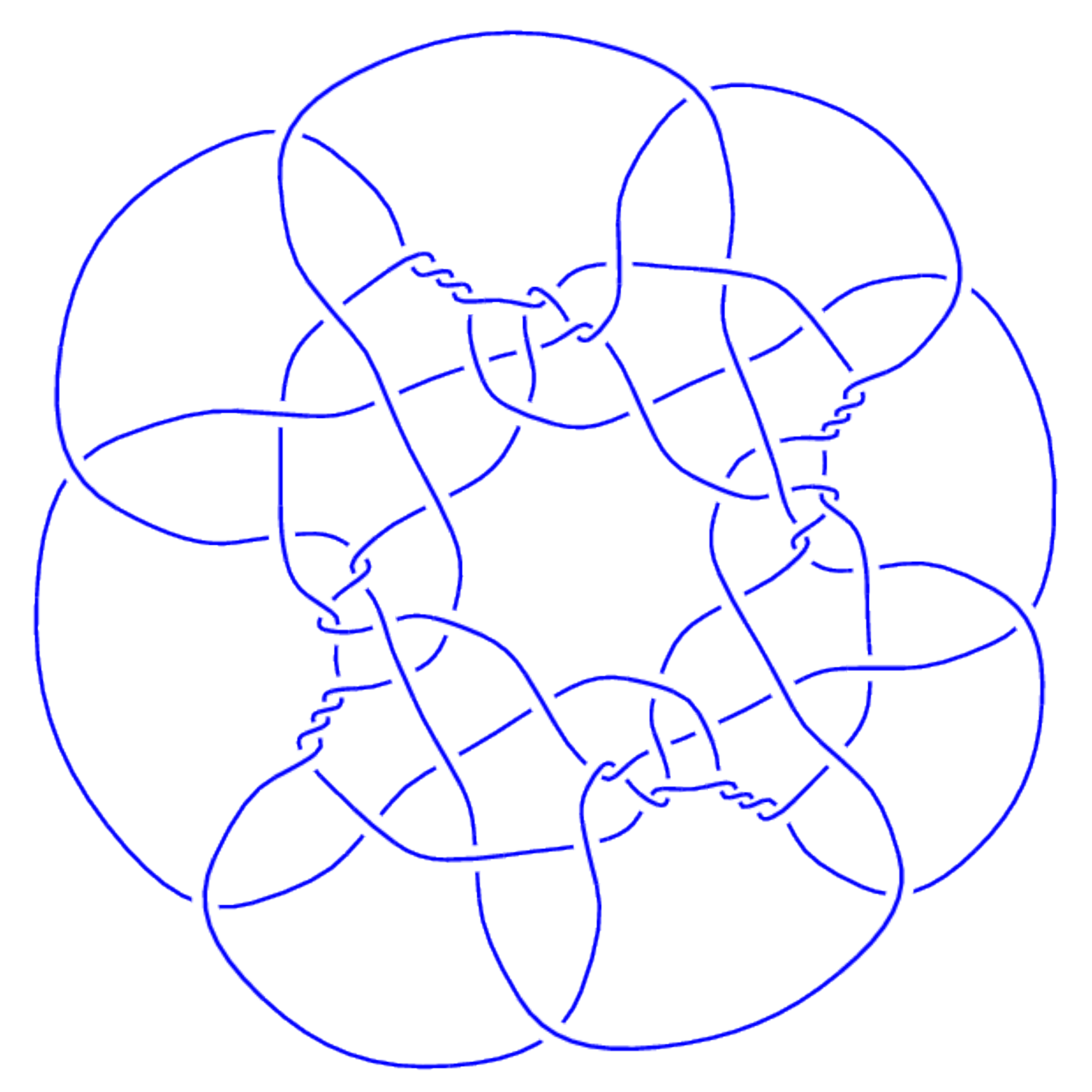}} &\qquad&	{\includegraphics[width=4cm]{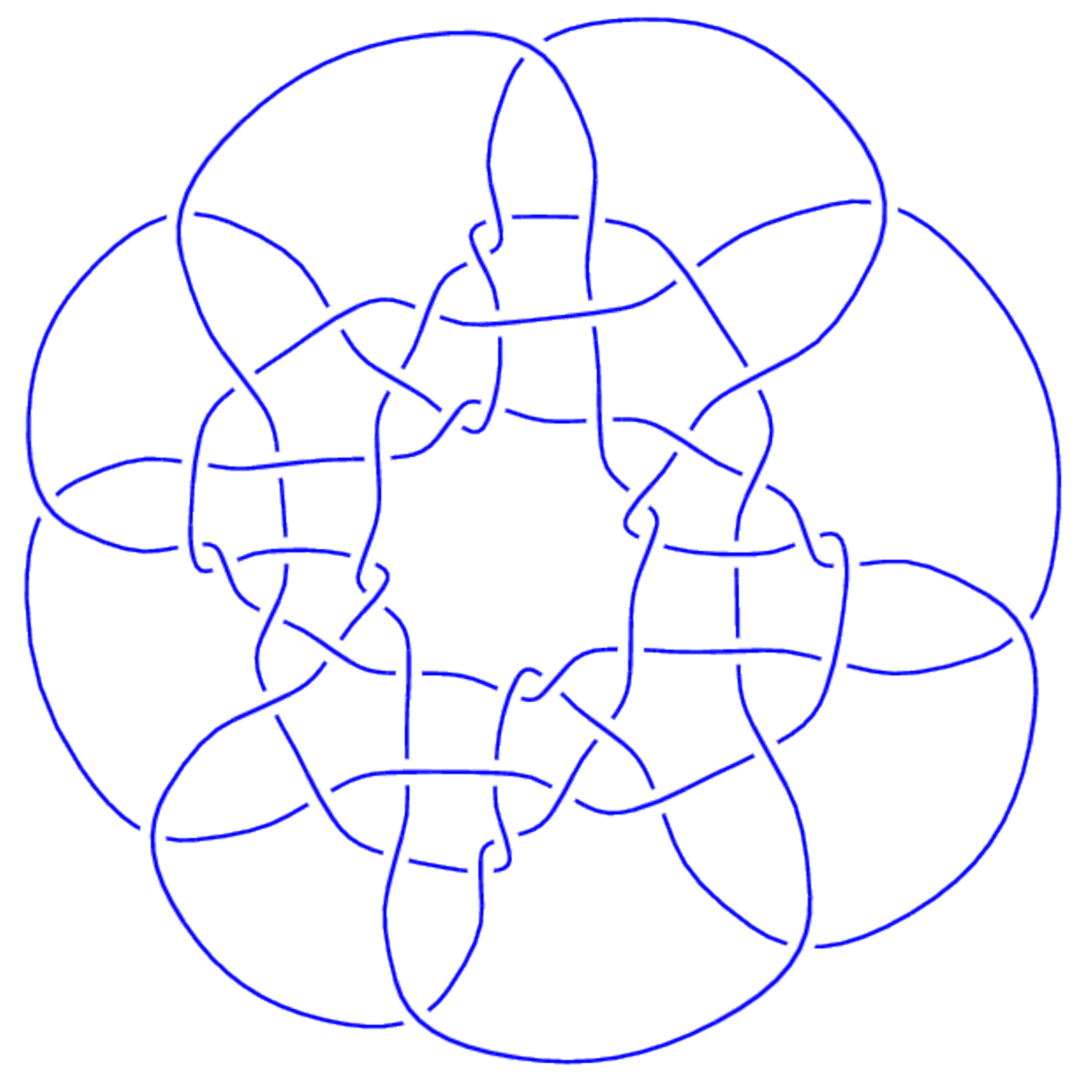}} \\
				Leading coefficient = 13 & \qquad&Leading coefficient = 17 \\	
\end{tabular}
\end{scriptsize}

\caption{Amphicheiral knots where the leading coefficient of the Conway polynomial is prime }\label{primek}
\end{figure}

\newpage
\section*{Acknowledgments}
The authors wish to thank Shashikant Mulay and Morwen Thistlethwaite for helpful discussions.

\end{document}